\documentclass[11pt]{amsart}
\usepackage{amsmath,amssymb,latexsym,esint,cite,mathrsfs}
\usepackage{verbatim,wasysym}
\usepackage[left=2.3cm,right=2.3cm,top=2.65cm,bottom=2.65cm]{geometry}
\usepackage{tikz,enumitem,graphicx, subfig, microtype, color}
\usepackage{epic,eepic}

\usepackage[colorlinks=true,urlcolor=blue, citecolor=red,linkcolor=blue,
linktocpage,pdfpagelabels, bookmarksnumbered,bookmarksopen]{hyperref}
\usepackage[hyperpageref]{backref}
\usepackage[english]{babel}
\allowdisplaybreaks
\numberwithin{equation}{section}
\newtheorem{theorem}{Theorem}[section]
\newtheorem{proposition}[theorem]{Proposition}

\newtheorem{lemma}[theorem]{Lemma}
\theoremstyle{definition}

\newtheorem{remark}[theorem]{Remark}

\newcommand{\R}{\mathbb{R}}

\begin{document}
\title
[Ground state solutions for zero-mass Chern-Simons-Schr\"{o}dinger systems]
{Generalized Chern-Simons-Schr\"{o}dinger system with critical
exponential growth: the zero-mass case}

 \author[L.\ Shen]{Liejun Shen}

\author[M.\ Squassina]{Marco Squassina}

 \address{Liejun Shen, \newline\indent Department of Mathematics, \newline\indent Zhejiang Normal University, \newline\indent
	Jinhua, Zhejiang, 321004, People's Republic of China}
\email{\href{mailto:ljshen@zjnu.edu.cn.}{ljshen@zjnu.edu.cn.}}

\address{Marco Squassina, \newline\indent
	Dipartimento di Matematica e Fisica \newline\indent
	Universit\`a Cattolica del Sacro Cuore, \newline\indent
	Via della Garzetta 48, 25133, Brescia, Italy}
\email{\href{mailto:marco.squassina@unicatt.it.}{marco.squassina@unicatt.it.}}


\thanks{L.J. Shen is partially supported by NSFC (12201565). M.\  Squassina  is  member  of  Gruppo  Nazionale  per
	l'Analisi  Matematica,  la Probabilita  e  le  loro  Applicazioni  (GNAMPA)  of  the  Istituto  Nazionale  di  Alta  Matematica  (INdAM)}

\begin{abstract}
We consider the existence of ground state solutions for a class of zero-mass Chern-Simons-Schr\"{o}dinger systems
\[
  \left\{
  \begin{array}{ll}
\displaystyle    -\Delta u +A_0 u+\sum\limits_{j=1}^2A_j^2 u=f(u)-a(x)|u|^{p-2}u, \\
 \displaystyle     \partial_1A_2-\partial_2A_1=-\frac{1}{2}|u|^2,~\partial_1A_1+\partial_2A_2=0,\\
 \displaystyle    \partial_1A_0=A_2|u|^2,~ \partial_2A_0=-A_1|u|^2,\\
  \end{array}
\right.
\]
where $a:\R^2\to\R^+$ is an external potential, $p\in(1,2)$
and $f\in \mathcal{C}(\R)$ denotes the nonlinearity that fulfills the
critical exponential growth in the Trudinger-Moser sense at infinity.
By introducing an improvement of
the version of Trudinger-Moser inequality approached in \cite{Shen2},
we are able to investigate the existence of positive ground state solutions for the given system using
variational method.
\end{abstract}
\maketitle

%
%\begin{center}
%	\begin{minipage}{8.5cm}
%		\small
%		\tableofcontents
%	\end{minipage}
%\end{center}
%
%\smallskip

\section{Introduction and main results}

In this article, we focus on establishing the existence of positive ground state solutions for the following
generalized Chern-Simons-Schr\"{o}dinger
(\textbf{CSS} in short) system/equation with critical exponential growth
\begin{equation}\label{mainequation1}
	\left\{
	\begin{array}{ll}
		\displaystyle    -\Delta u +A_0 u+\sum\limits_{j=1}^2A_j^2 u= f(u)-a(x)|u|^{p-2}u, \\
		\displaystyle     \partial_1A_2-\partial_2A_1=-\frac{1}{2}|u|^2,~\partial_1A_1+\partial_2A_2=0,\\
		\displaystyle    \partial_1A_0=A_2|u|^2,~ \partial_2A_0=-A_1|u|^2,
	\end{array}
	\right.
\end{equation}
where $a:\R^2\to\R^+$ is an external potential, $p\in(1,2)$
and $f\in \mathcal{C}(\R)$ denotes the nonlinearity that fulfills the
critical exponential growth in the Trudinger-Moser sense at infinity which would be specified later.

Recently, great attention has been paid to
the time-dependent CSS system in two spatial dimension
\begin{equation}\label{CSS1}
	\begin{cases}
		\text{i}D_0\psi+(D_1D_1+D_2D_2)\psi+g(x,|\psi|^2)\psi=0, \\
		\partial_0A_1-\partial_1A_0=-\text{Im}(\overline{\psi} D_2\psi),\\
		\partial_0A_2-\partial_2A_0=\text{Im}(\overline{\psi} D_1\psi), \\
		\partial_1A_2-\partial_2A_1=-\frac{1}{2}|\psi|^2,\\
	\end{cases}
\end{equation}
where i stands for the imaginary unit,
$\partial_0=\frac{\partial}{\partial t}$, $\partial_1=\frac{\partial}{\partial x_1}$, $\partial_2=\frac{\partial}{\partial x_2}$
for $(t,x_1,x_2)\in\R^{1+2}$, $\psi:\R^{1+2}\to \mathbb{C}$ acts as the complex scalar field,  $A_j:\R^{1+2}\to\R$ denotes the gauge field,
$D_j=\partial_j +\text{i}A_j$ is the covariant derivative for $j=0, 1, 2$
and $g$ is the nonlinearity. In real world, it
is usually exploited to describe the non-relativistic dynamics behavior
of massive number of particles in Chern-Simons gauge fields.
This model plays an important role in
the study of high-temperature superconductors, Aharovnov-Bohm scattering, and quantum Hall effect,
we refer the reader to \cite{Jackiw1,Jackiw2,Jackiw3}. Moreover, there exist some further physical motivations
for considering CSS system \eqref{CSS1}, see
\cite{Dunne,Huh1,LS1,LS2} for example.

For all $(t,x_1,x_2)\in\R^{1+2}$ and $j=0,1,2$, one usually
  considers the situation $A_j(t,x)=A_j(x)$.
If the standing wave ansatz $\psi(t,x)=e^{\text{i}\lambda t}u(x)$ with a given $\lambda\in\R$
for $u:\R^2\to\R$, then \eqref{CSS1} reduces to
\begin{equation}\label{CSS2a}
	\left\{
	\begin{array}{ll}
		\displaystyle  -\Delta u+\lambda u+A_0 u+\sum_{j=1}^2A_j^2 u=f(x,u),  \\
		\displaystyle     \partial_1A_2-\partial_2A_1=-\frac{1}{2}|u|^2,\\
		\displaystyle    \partial_1A_0=A_2|u|^2,~ \partial_2A_0=-A_1|u|^2,
	\end{array}
	\right.
\end{equation}
where $f(x,u)=g(x,|u|^2)u$. Suppose $A_j$ satisfies the Coulomb gauge condition
$\sum_{j=0}^2\partial_jA_j= 0$, then
\eqref{CSS2a} with $\lambda\equiv0$ becomes the original CSS equation \eqref{mainequation1}, namely
\begin{equation}\label{CSS2b}
	\left\{
	\begin{array}{ll}
		\displaystyle    -\Delta u+A_0 u+\sum_{j=1}^2A_j^2 u=f(x,u),  \\
		\displaystyle    \partial_1A_0=A_2|u|^2,~ \partial_2A_0=-A_1|u|^2,\\
		\displaystyle  \partial_1A_2-\partial_2A_1=-\frac{1}{2}|u|^2,~
		\partial_1A_1+\partial_2A_2 =0.
	\end{array}
	\right.
\end{equation}
It follows from $\partial_1A_0=A_2|u|^2$ and $\partial_2A_0=-A_1|u|^2$
in \eqref{CSS2b} that
\[
\Delta A_0 =\partial_1\big(A_2|u|^2\big)-\partial_2\big(A_1|u|^2\big),
\]
leading to
\begin{equation}\label{CSS2d}
	A_0[u](x)=\frac{x_1}{2\pi |x|^2}\ast \big(A_2|u|^2\big)
	-\frac{x_2}{2\pi |x|^2}\ast \big(A_1|u|^2\big).
\end{equation}
In a similar way, we depend on
  $\partial_1A_2-\partial_2A_1=-\frac{1}{2}|u|^2$ and
$\partial_1A_1+\partial_2A_2 =0$ in \eqref{CSS2b} to derive
\[
\Delta A_1= \partial_2\bigg(\frac{|u|^2}{2}\bigg)~ \text{and}~
\Delta A_2= -\partial_1\bigg(\frac{|u|^2}{2}\bigg).
\]
From which,
the components $A_j$ for $j=1,2$ in \eqref{CSS2b} can be represented as    %\label{CSS2e1}\label{CSS2e2}
\begin{equation}\label{CSS2e}
   \left\{
     \begin{array}{ll}
   \displaystyle    A_1[u](x)=\frac{x_2}{2\pi|x|^2}\ast\bigg(\frac{|u|^2}{2}\bigg)
	=-\frac{1}{4\pi }\int_{\R^2}\frac{(x_2-y_2)u^2(y)}{|x-y|^2}dy,\\
    \displaystyle      A_2[u](x) =-\frac{x_1}{2\pi|x|^2}\ast\bigg(\frac{|u|^2}{2}\bigg)
	=\frac{1}{4\pi }\int_{\R^2}\frac{(x_1-y_1)u^2(y)}{|x-y|^2}dy.
     \end{array}
   \right.
\end{equation}
In the sequel, we shall write $A_j$ in place of $A_j[u]$ for $j\in\{0,1,2\}$ for simplicity
as long as there is no misunderstanding. There are some
further properties of $A_j$ for $j\in\{0,1,2\}$ in Section \ref{Sec2} below.

Indeed, CSS system \eqref{CSS1} can reduce to a single equation
if one studies the standing wave ansatz $\psi(t,x)=e^{\text{i}\lambda t}u(x)$
with a radially symmetric $u$.
Actually,
 Byeon-Huh-Seok \cite{Byeon} considered the standing waves of type
\begin{equation}\label{CSS2}
	\begin{gathered}
		\psi(t,x)=u(|x|)e^{\text{i}\lambda t},~  A_0(t,x)=k(|x|), \hfill\\
		A_1(t,x)=\frac{x_2}{|x|^2}h(|x|),~   A_2(t,x)=-\frac{x_1}{|x|^2}h(|x|),\hfill\\
	\end{gathered}
\end{equation}
where $k$ and $h$ are real value functions depending only on $|x|$. Note that \eqref{CSS2} satisfies the Coulomb gauge condition with $\varsigma=ct+n\pi$, where $n$ is
an integer and $c$ is a real constant. To seek for solutions of CSS system \eqref{CSS1} of the type \eqref{CSS2},
it is enough to handle the
following semilinear elliptic equation
\begin{equation}\label{BHS}
	-\Delta u+ \lambda u+ \bigg(\int_{|x|}^\infty
	\frac{h(s)}{s}u^2(s)ds+\frac{h^2(|x|)}{|x|^2} \bigg)u=f(x,u)
	~\text{in}~\mathbb{R}^2,
\end{equation}
where $h(s)=\int_0^s\frac{r}{2}u^2(r)dr$. As before, we continue to assume that $\lambda\equiv0$
in Eq. \eqref{BHS}.

At this stage, there are two kinds of CSS equations,
\eqref{CSS2b} and \eqref{BHS}, which could be called by the so called zero-mass ones.
Generally, when $f(x,t)=\bar{f}(x,t)-V(x)t$ for all $(x,t)\in\R^2\times\R$
in the classic CSS equations, more and more
  interesting results have been explored by many mathematicians
over the past decades for various assumptions on $\bar{f}$ and $V$.
Speaking precisely, for $\bar{f}(x,t)=|t|^{p-2}t$ and $V\equiv1$, by
exploiting the Nehari-Poho\u{z}aev manifold argument,
Byeon \emph{et al.} \cite{Byeon} derived the existence of positive solutions for all $p>6$.
Particularly, with the prescribed mass constraint $\int_{\R^2}|u|^2dx=c^2$,
they showed some existence results
for each $c\neq0$ if $p\in(2,3]$ and sufficiently small $|c|$ if $p\in(3,4)$. Afterwards,
the existence, nonexistence and multiplicity of nontrivial solutions for \eqref{CSS2a}, or \eqref{BHS},
have been considerably contemplated by a lot of mathematicians, see \cite{AP,CZT,KT,LOZ,LYY,PR,Shen,SS,SSY,PSZZ,WT,DPS,JCZSA}
and the references therein for example even if these references are far to be exhaustive.

Next, we should turn to consider the so-called
zero-mass CSS equation.
Very recently, Zhang, Tang and Chen \cite{ZTC}
handled the following zero-mass CSS equation
\begin{equation}\label{zeromass}
-\Delta u+ \bigg(\int_{|x|}^\infty
	\frac{h(s)}{s}u^2(s)ds+\frac{h^2(|x|)}{|x|^2} \bigg)u=f(u)-\bar{a}|u|^{p-2}u
	~\text{in}~\mathbb{R}^2,
\end{equation}
where $\bar{a}>0$ is a constant, $p\in(3,4)$
and the nonlinearity $f$ admits the critical exponential growth
in the Trudinger-Moser sense at infinity. In fact, we say that
a function $f$ possesses the
\emph{critical exponential growth} at infinity if there exists a constant $\alpha_{0}>0$ such that
\begin{equation}\label{definition}
	\lim _{t \rightarrow+\infty} \frac{|f(t)|}{e^{\alpha t^{2}}}= \begin{cases}0,
		& \forall \alpha>\alpha_{0}, \\ +\infty, & \forall \alpha<\alpha_{0}.\end{cases}
\end{equation}
The above definition was introduced by Adimurthi and Yadava in \cite{AYA}, see also de Figueiredo, Miyagaki
and Ruf \cite{Figueiredo} for example.

In \cite{ZTC}, the authors depended on
the work space below
 \[
 E\triangleq \bigg\{u:u(x)~\text{is Lebesgue measurable s.t.}\int_{\R^2}|\nabla u|^2dx<+\infty
 ~\text{and}\int_{\R^2}|u|^pdx<+\infty\bigg\}
 \]
which is the completion of $C_0^\infty(\R^2)$ under the norm
 \[
 \|u\|=\sqrt{|\nabla u|_2^2+|u|_p^2},~\forall u\in E,
 \]
 where $|\cdot|_q$ denotes the usual norm corresponding to the
Lebesgue space
  $L^q(\R^2)$ for every $1\leq q\leq\infty$.
In order to treat the problem variationally, proceeding as \cite{AYA,AY,Cao,Figueiredo,DSD,Lu2012,Yang0},
they established the following version of Trudinger-Moser inequality
 \begin{proposition}\label{ZTCTM}
 Suppose that $3<p<4$, then
 $(e^{\alpha u^2}-1-\alpha u^2)\in L^1(\R^2)$ for all $\alpha>0$ and $u\in E$. Moreover,
if $u\in E$, $|\nabla u|_2^2\leq1$, $|u|_p^p\leq M>+\infty$ and $\alpha<4\pi$, then there exists a constant
$C(M,\alpha)>0$, which depends only on $M$ and $\alpha$, such that
 \begin{equation}\label{ZTCTM1}
 \int_{\R^2}(e^{\alpha u^2}-1-\alpha u^2)dx\leq C(M,\alpha).
 \end{equation}
\end{proposition}
\noindent With the help of Proposition \ref{ZTCTM}, they concluded the existence of mountain-pass solutions for
Eq. \eqref{zeromass} with a nonlinearity $f$ involving the critical exponential growth.
Actually, to search for the nontrivial solutions, they
Actually, to look for the nontrivial solutions, they restricted
themselves in the radially symmetric subspace of $E$, namely $E_r=\{u\in E:u(x)=u((|x|)\}$.
In this situation, they immediately have the compact imbedding
$E_r\hookrightarrow L^s(\R^2)$ for all $p<s<+\infty$.

Afterwards, Shen \cite{Shen2} generalized and improved the results in \cite{ZTC}
to the case that $1<p<2$ and the nonlinearity $f$ having supercritical exponential growth.
Precisely, by contemplating
  the work space $E$ above
and introducing the
the Young function defined by
 \begin{equation}\label{Youngfunction}
\Phi_{\alpha,j_0}(t)=e^{\alpha t^2}-\sum_{j=0}^{j_0-1}\frac{\alpha^j}{j!}|t|^{2j},~\forall t\in\R,
 \end{equation}
where
 $\alpha>0$ appearing in \eqref{definition} and $j_0\triangleq \inf\{j\in \mathbb{N}^+:2j\geq p^*\}$ with
$p^*=\frac{2p}{2-p}>2$,
Shen \cite{Shen2} firstly established the Trudinger-Moser inequality below
 \begin{proposition}\label{ShenTM}
 Suppose that $1<p<2$, then
 $\Phi_{\alpha,j_0}(u)\in L^1(\R^2)$ for all $\alpha>0$ and $u\in E$. Moreover
 \begin{equation}\label{ShenTM1}
\mathbb{S}(\alpha)\triangleq \sup_{u\in E,\|u\|\leq1}\int_{\R^2}\Phi_{\alpha,j_0}(u) dx<+\infty
 \end{equation}
for all $0<\alpha<4\pi$. Finally, if $\alpha>4\pi$, then $\mathbb{S}(\alpha)=+\infty$.
\end{proposition}
\noindent Then, combining the minimax procedure and elliptic regular theory,
Shen investigated the existence of a nontrivial solution
with the mountain-pass energy in \cite{Shen2}, where the subspace $E_r$ was still considered.

Motivated by all of the quoted papers above, particular by \cite{ZTC,Shen2},
it is quite natural to ask some interesting questions. For example,

(I) As pointed out in \cite{Shen2}, either \eqref{ZTCTM1} or \eqref{ShenTM1} is
a subcritical Trudinger-Moser type inequality in the whole space $\R^2$,
namely there is no information when
$\alpha$ is exactly equal to $4\pi$.
Thereby, can we given an affirmative answer that whether $\mathbb{S}(4\pi)<+\infty$ or $\mathbb{S}(4\pi)=+\infty$.

(II) Owing to the compact imbedding
$E_r\hookrightarrow L^s(\R^2)$ for each $p<s<+\infty$, although the nonlinearity $f$ possesses the (super)critical exponential growth
in \cite{ZTC,Shen2}, it is simple to recover the compactness to some extent.
Hence, can we find nontrivial solutions for zero-mass CSS equation \eqref{zeromass}.
In other words, whether the existence results in \cite{ZTC,Shen2} remain true for Eq. \eqref{CSS2b}
with $f(x,t)=\bar{a}|t|^{p-2}t$ for all $x\in \R^2$ and $t\in\R$, where $p\in(1,2)$, or $p\in(3,4)$.

(III) The reader is invited to observe that Eq. \eqref{zeromass} is an autonomous one
because $\bar{a}>0$ is just a constant. Thus, can we improve this constant to a general potential function.
Moreover, if it was true, whether the obtained nontrivial solution is indeed a ground state solution.

As a consequence, we shall try our best to introduce some new analytic tricks and then
contemplate the above Questions.

First of all, we focus on the Question (I).
Let us continue to use the Young function
$\Phi_{\alpha,j_0}$ defined in \eqref{Youngfunction},
we shall prove the following result.

 \begin{theorem}\label{maintheorem1}
 Suppose that $1<p<2$, then
 $\Phi_{\alpha,j_0}(u)\in L^1(\R^2)$ for all $\alpha>0$ and $u\in E$. Moreover
 \begin{equation}\label{maintheorem1a}
\mathbb{S}(\alpha)\triangleq \sup_{u\in E,\|u\|\leq1}\int_{\R^2}\Phi_{\alpha,j_0}(u) dx<+\infty
 \end{equation}
for all $\alpha\in(0,4\pi]$. Moreover, $\mathbb{S}(\alpha)=+\infty$
 if $\alpha>4\pi$.
\end{theorem}

\begin{remark}
Due to Theorem \ref{maintheorem1}, we can make sure that $\mathbb{S}(4\pi)<+\infty$,
and so it solves the Question (I) completely. Moreover, we do believe that the technique for the proof of
Theorem \ref{maintheorem1} can be also adapted to Proposition \ref{ZTCTM}.
\end{remark}

Next, in order to solve Questions (II)
and (III), we are ready to introduce some technical assumptions on the potential $a:\R^2\to\R$
and the nonlinearity $f:\R\to\R$ in Eq. \eqref{mainequation1}
as follows.
\begin{itemize}
  \item[$(A_1)$] $a\in \mathcal{C}^0(\R^2)$ with $\inf\limits_{x\in\R^2}a(x)>0$;
  \item[$(A_2)$] for almost every $x\in\R^2$, $a(x)\leq \lim\limits_{|x|\to\infty}a(x)\triangleq a_\infty<+\infty$ and this inequality is
strict in a subset of positive Lebesgue measure
\end{itemize}
Concerning the nonlinearity $f$, we suppose that
\begin{itemize}
	\item[$(f_1)$] $f \in \mathcal{C}(\mathbb{R})$ with $f(t)\equiv0$ for all $t\in(-\infty,0]$ and $f(t)=o(t)$ as $t\to0^+$;
	\item[$(f_2)$] the map $t\mapsto f(t)/t^5$ is strictly increasing on $t\in(0,+\infty)$;
    \item[$(f_3)$] there exist constants $t_0 > 0$, $M_0 > 0$ and $\vartheta \in[0, 1)$ such that
$$
0<t^{\vartheta}F(t)\leq M_0f(t),~\forall t>t_0,
$$
where and in the sequel $F(t)=\int_0^tf(s)ds$ for all $t>0$;
    \item[$(f_4)$] $\lim\limits_{t\to+\infty}F(t)e^{-\alpha_0t^2}\triangleq\beta_0>0$, where $\alpha_0>0$ comes from \eqref{definition}.
\end{itemize}

We are now in a position to state the second main result in this article.
 \begin{theorem}\label{maintheorem2}
 Let $1<p<2$ and suppose $(A_1)-(A_2)$. If $f$ satisfies \eqref{definition} and $(f_1)-(f_4)$, then
Eq. \eqref{mainequation1} admits at least a positive ground state solution in $E$.
\end{theorem}

\begin{remark}
It is obvious that Questions (II)
and (III) are uncovered by Theorem \ref{maintheorem2}
which in turn indicates that our results improve and replenish the counterparts in
\cite{ZTC,Shen2}. It should be mentioned here that both
the assumptions on the potential $a$ and the nonlinearity $f$
are standard. On the one hand, the function $a$ equipping with $(A_1)-(A_2)$ is usually called by the well-known Rabinowitz's type
potential introduced in \cite{Rabinowitz} and it was later
exploited by Wan and Tan in \cite{WT}. On the other hand,
as to the function $f$ having critical exponential growth and satisfying
 $(f_1)-(f_4)$, we prefer to refer the reader to \cite{PSZZ,SS,ZTC} and their references therein.
\end{remark}

Finally, we shall exhibit the main idea for the proof of Theorem \ref{maintheorem2}.
The reader is invited to see that
the work space
\[
 E_a\triangleq \bigg\{u:u(x)~\text{is Lebesgue measurable s.t.}\int_{\R^2}|\nabla u|^2dx<+\infty
 ~\text{and}\int_{\R^2}a(x)|u|^pdx<+\infty\bigg\}
 \]
endowed with the norm
$$
\|u\|_a=\left(\int_{\R^2}|\nabla u|^2dx+ \left(\int_{\R^2}a(x)|u|^pdx\right)^{\frac2p}\right)^{\frac12},~\forall u\in E_a,
$$
is equivalent to $(E,\|\cdot\|)$
because $a$ is a positive and bounded function in $\R^2$.
Thus, we will exploit the work space $(E,\|\cdot\|)$,
instead of $(E_a,\|\cdot\|_a)$, just for simplicity.
Due to the lack of compactness caused by the critical exponential growth and the absence of the compact imbedding
$E_r\hookrightarrow L^s(\R^2)$ for every $p<s<+\infty$,
 the foremost point of the proof of Theorem \ref{maintheorem2}
is to restore the compactness.
Inspired by \cite{WT},
we need to investigate the existence of ground
state solutions of the associated ``limit problem" of \eqref{mainequation1}, which is given as
\begin{equation}\label{mainequation2}
	\left\{
	\begin{array}{ll}
		\displaystyle    -\Delta u+a_\infty |u|^{p-2}u +A_0 u+\sum\limits_{j=1}^2A_j^2 u= f(u), \\
		\displaystyle     \partial_1A_2-\partial_2A_1=-\frac{1}{2}|u|^2,~\partial_1A_1+\partial_2A_2=0,\\
		\displaystyle    \partial_1A_0=A_2|u|^2,~ \partial_2A_0=-A_1|u|^2.
	\end{array}
	\right.
\end{equation}

We obtain the following result.
 \begin{theorem}\label{maintheorem3}
 Let $1<p<2$. If $f$ satisfies \eqref{definition} and $(f_1)-(f_4)$, then
Eq. \eqref{mainequation2} possesses a positive ground state solution in $E$.
\end{theorem}

\begin{remark}
Obviously, one realizes that Theorem \ref{maintheorem3} also provides a positive answer to the Question (II)
above. In the proof of Theorem \ref{maintheorem3}, the most striking point is that we success in establishing the Vanishing lemma
corresponding to the work space $(E,\|\cdot\|)$, see Theorem \ref{Vanishing} below. Although the essential idea originates
from its classic version due to Lions, c.f. \cite[Lemma 1.21]{Willem},
we have to make some efforts to prove it and it may prompt some further studies
for zero-mass Schr\"{o}dinger equation.
\end{remark}

With Theorem \ref{maintheorem3} in hands, the solvability of Theorem \ref{maintheorem2}
 becomes available so far, but we emphasize here that the condition $(f_2)$
plays a crucial role in restoring the compactness,
see Lemma \ref{unique} and \eqref{equality} for instance.
As a consequence, one naturally wonders that
whether there still exists a mountain-pass type solution
for Eq. \eqref{mainequation2}
when $(f_2)$ is replaced with a weak type condition below
\begin{itemize}
	\item[$(f'_2)$] for all $t>0$, there holds $f(t)t-6F(t)\geq0$.
\end{itemize}

Actually, we are going to conclude the existence result as follows.

 \begin{theorem}\label{maintheorem4}
 Let $1<p<2$ and suppose $(A_1)-(A_2)$. If $f$ satisfies \eqref{definition} and $(f_1)-(f'_2)$ as well as $(f_3)-(f_4)$, then
Eq. \eqref{mainequation2} has a positive mountain-pass type solution in $E$.
\end{theorem}

\begin{remark}
We note that Theorem \ref{maintheorem4} solves Question (II) and (III)
partially.
Let us point out here that we will borrow some idea adopted in \cite{Shen2}
to reach the proof. Nevertheless, there are some new challenges
that prevent us repeating the arguments simply. For example,
as to the critical exponential case in \cite{Shen2}, the author strongly relied on the following condition of type
\begin{itemize}
	\item[$(f'_4)$] there are $\gamma>0$ and $s>6$ such that $F(t)\geq\gamma t^s$ for all $t\geq0$,
\end{itemize}
to restore the compactness, where $\gamma>0$ is sufficiently large.
One would easily deduce that the condition $(f_4)$ in the present article is weaker than $(f'_4)$
which is a global one that does never reveal the
essential feature of the critical exponential growth in \eqref{definition}.
\end{remark}

The outline of the paper is organized as follows. In Section \ref{Sec2}, we mainly
present some preliminary results and show the proofs of Theorem \ref{maintheorem1}.
Sections \ref{Sec3} and \ref{Sec4} are devoted to the proofs of Theorems \ref{maintheorem3} and \ref{maintheorem2}, respectively.
The proof of Theorem \ref{maintheorem4} shall be presented in Section \ref{Sec5}.
\\\\
\noindent \textbf{Notations:} From now on in this paper, otherwise mentioned, we ultilize the following notations:
\begin{itemize}
	\item   $C,C_1,C_2,\cdots$ denote any positive constant, whose value is not relevant and $\R^+\triangleq(0,+\infty)$.
	 	\item      Let $(X,\|\cdot\|_X)$ be a Banach
space with dual space $(X^{-1},\|\cdot\|_{X^{-1}})$, and $\Psi$ be functional on $X$.
	\item The (C) sequence at a level $c\in\R$ ($(C)_c$ sequence in short)
corresponding to $\Psi$ means that $\Psi(x_n)\to c$ and $(1+\|x_n\|_X)\|\Psi^{\prime}(x_n)\|_{X^{-1}}\to 0$ in $\R$ as $n\to\infty$, where
$\{x_n\}\subset X$.
	\item  $|\cdot|_p$ stands for the usual norm of the Lebesgue space $L^{p}(\mathbb{R}^2)$ for all $p \in [1,+\infty]$.
\item For any $\varrho>0$ and every $x\in \R^2$, $B_\varrho(x)\triangleq\{y\in \R^2:|y-x|<\varrho\}$.
\item $o_{n}(1)$ denotes the real sequences with $o_{n}(1)\to 0$
 as $n \to +\infty$.
\item $``\to"$ and $``\rightharpoonup"$ stand for the strong and
 weak convergence in the related function spaces, respectively.
\end{itemize}

\section{Variational Framework and Preliminaries}\label{Sec2}
In this section, we are going to exhibit some preliminary results which enable us
to treat the problems variationally.

First of all, let us recall some imbedding results which would play a foremost role in
formulating the variational structure.
The following results can be found in \cite[Lemmas 2.1 and 2.2]{Shen2},
so we shall omit the detailed proofs.

\begin{lemma}\label{imbedding}
Assume $1<p<2$, then the imbedding $E \hookrightarrow L^s(\R^2)$ is continuous
and $E \hookrightarrow L^s_{\text{loc}}(\R^2)$ is compact for all $p\leq s<+\infty$,  respectively.
Moreover, $(E,\|\cdot\|)$ is a reflexive Banach space.
\end{lemma}

Then, we turn to contemplate the so called Chern-Simons term in Eq. \eqref{mainequation1}.
To begin with, there exist some meaningful and significant observations. According to
 the second equation and the last two equations in Eq. \eqref{mainequation1}, for each $u\in E$, one has
\begin{equation}\label{gauge0}
\begin{gathered}
\int_{\mathbb{R}^2}A_0|u|^2dx  =2\int_{\mathbb{R}^2}A_0(\partial_2A_1-\partial_1A_2)dx   \hfill\\
 \ \ \ \   =2\int_{\mathbb{R}^2}(A_2\partial_1A_0-A_1\partial_2A_0)dx
 =2\int_{\mathbb{R}^2}(A_1^2+A_2^2)|u|^2dx.\hfill\\
\end{gathered}
\end{equation}

As a by-product of
the well-known Hardy-Littlewood-Sobolev inequality \cite[Theorem 4.3]{LM}, we could conclude
the following estimates to the gauge fields $A_j$ for $j\in\{0, 1, 2\}$.

\begin{lemma}\label{gauge}
(see \cite[Propositions 4.2-4.3]{Huh2}) Assume $1<r<2$ and $\frac{1}{r}
-\frac{1}{\widehat{r}}=\frac{1}{2}$, then
\[
|A_j|_{\widehat{r}}\leq C_r|u|_{2r}^2~ \text{for}~ j=1,2,~
|A_0|_{\widehat{r}}\leq C_r|u|_{2r}^2|u|_4^2,
\]
where $C_r>0$ is a constant dependent of $r$.
\end{lemma}

Combining Lemmas \ref{imbedding} and \ref{gauge}, one can easily see that
\begin{equation}\label{gauge1}
  |A_ju|_2\leq |A_j|_{\widehat{r}}|u|_{\frac{r}{r-1}}
\leq C_r|u|_{2r}^2|u|_{\frac{r}{r-1}}\leq \bar{C}_r\|u\|^3,
~ \text{for}~ j=1,2,
\end{equation}
because $2r>2$ and $r/(r-1)>2$, where $\bar{C}_r>0$  depends only on  $r>1$.
We also need the following Br\'{e}zis-Lieb type lemma for the Chern-Simons term.
\begin{lemma}\label{BLCS}
(see \cite[Lemma 2.4]{GZ})
If $u_n\rightharpoonup u$ in $E$ and $u_n\to u$ a.e. in $\R^2$, then one has
$A_j [u_n]\to A_j[u]$ a.e. for $j=1,2$,
\begin{equation}\label{BLCS1}
 \left\{
   \begin{array}{ll}
  \displaystyle    \lim_{n\to\infty}\int_{\R^2}A_0[u_n]u_n\psi dx=\int_{\R^2}A_0[u ]u \psi dx,~\forall \psi\in E,\\
   \displaystyle      \lim_{n\to\infty}\int_{\R^2}A_j^2[u_n]u_n\psi dx=\int_{\R^2}A_j^2[u ]u \psi dx,
   ~\forall \psi\in E~\text{\emph{with}}~j=1,2,
   \end{array}
 \right.
\end{equation}
 and
\begin{equation}\label{BLCS2}
\lim_{n\to\infty}\int_{\R^2}\big[A_j^2[u_n]|u_n|^2-A_j^2[u_n-u]|u_n-u|^2\big]dx
=\int_{\R^2}A_j^2[u]|u|^2dx,~\text{for}~j=1,2.
\end{equation}
\end{lemma}

Finally, we shall focus on the nonlinearity $f$.
Whereas, it possesses the critical exponential growth at infinity in the Trudinger-Moser sense at infinity,
we have to derive the Trudinger-Moser type inequality associated with the work space $E$
in this article.

\begin{proof}[\textbf{\emph{Proof of Theorem \ref{maintheorem1}}}]
Let $\alpha\in(0,4\pi]$ and $u\in E$ with $\|u\|\leq1$. We denote by $u^*$
the Schwarz symmetrization of $u$, then $u^*$ is radial and non-increasing.
Thanks to the results in \cite{Kawohl},
\[
\int_{\R^2} |\nabla u^*|^2 dx\leq\int_{\R^2} |\nabla u|^2 dx,~
 \int_{\R^2} |u^*|^p dx=\int_{\R^2} |u|^p dx
\]
and
\[
\int_{\R^2} \Phi_{\alpha,j_0}(u^*)dx=\int_{\R^2} \Phi_{\alpha,j_0}(u)dx.
\]
So, without loss of generality, we can suppose that $u\in E$ is non-increasing.
Given an $R>0$ which will be determined later,
due to Lemma \ref{imbedding}, we could deduce that the function $v(x)\triangleq u(x)-u(R)$
belongs to $H_0^1(B_R(0))$. Adopting the Young's inequality, there holds
\begin{align*}
 u^2(x) & =v^2(x)+2v(x)u(R)+u^2(R)\leq v^2(x)+[1+v^2(x)u^2(R)]+u^2(R) \\
   & =[1+u^2(R)]v^2(x)+1+u^2(R)\triangleq w^2(x)+1+u^2(R),~\forall x\in B_R(0).
\end{align*}
Obviously, $w\triangleq\sqrt{1+u^2(R)}v\in H_0^1(B_R(0))$ and $\nabla w=\sqrt{1+u^2(R)}\nabla v=\sqrt{1+u^2(R)}\nabla u$.
Moreover, we recall from \cite[Lemma A.IV]{Berestycki1},
because $u\in L^{p}(\R^2)$ by Lemma \ref{imbedding}, there is a constant $C_p>0$ which is independent of $u$ such that
\begin{equation}\label{Strauss1}
|u(x)|\leq C_p|x|^{-\frac2p}|u|_p~\text{for all}~x\neq0.
  \end{equation}
Consequently, with the help of \eqref{Strauss1}, we are able to see that
\begin{align*}
  \int_{\R^2}|\nabla w|^2dx &  =[1+u^2(R)]\int_{\R^2}|\nabla u|^2dx
  \leq[1+u^2(R)]\left[1-\left(\int_{\R^2}|u|^pdx\right)^{\frac2p}\right] \\
    & \leq1-\left(\int_{\R^2}|u|^pdx\right)^{\frac2p}+u^2(R)
    \leq1-(1-C_p^2R^{-\frac4p})\left(\int_{\R^2}|u|^pdx\right)^{\frac2p}.
\end{align*}
We then choose the constant $R$ to be sufficiently large such that $1-C_p^2R^{-\frac4p}\geq0$ and so
 $|\nabla w|_2^2\leq1$. As a consequence,
we can apply the classic Trudinger-Moser inequality explored in \cite{Cao,DSD} to arrive at
$$
\int_{B_R(0)}\Phi_{4\pi,j_0}(u)\mathrm{d}x\leq \int_{B_R(0)}e^{4\pi u^2} \mathrm{d}x
\leq e^{4\pi (1+u^2(R))}\int_{B_R(0)}e^{4\pi w^2} \mathrm{d}x\leq C.
$$
From which, one concludes that
\begin{equation}\label{maintheorem1d}
\sup_{u\in E:\|u\|\leq1}\int_{B_R(0)}\Phi_{4\pi,j_0}(u)\mathrm{d}x\leq C<+\infty.
\end{equation}

On the other hand,
we can follow the proof of \cite[Theorem 1.1]{Shen2} to verify that  the integral on the complement of $B_R(0)$
is uniformly bounded. For the sake of the reader's convenience, we shall show the proof in detail. Using
\cite[Lemma A.IV]{Berestycki1} and Lemma \ref{imbedding} again,
for each $u\in L^{p^*}(\R^2)$, there is a constant $\bar{C}_p>0$ independent of $u$ such that
\begin{equation}\label{Strauss}
|u(x)|\leq \bar{C}_p|x|^{-\frac{2-p}{p}}|u|_{p^*}~\text{for all}~x\neq0.
  \end{equation}
Since $j_0=\inf\{j\in \mathbb{N}:2j\geq p^*\}$, then
one sees $2j\geq p^*$ for all $j\geq j_0$ and so applying \eqref{Strauss} and $R>1$,
\[
\begin{gathered}
 \int_{\R^2\backslash B_R(0)}\Phi_{4\pi,j_0}(u)\mathrm{d}x   =\sum_{j=j_0}^{\infty}
\frac{(4\pi)^j }{j!}\int_{\R^2\backslash B_R(0)}|u|^{2j}\mathrm{d}x
=\sum_{j=j_0}^{\infty}
\frac{(4\pi)^j }{j!}\int_{\R^2\backslash B_R(0)}|u|^{2j-p^*}|u|^{p^*}\mathrm{d}x\hfill\\
\ \ \ \   \leq \sum_{j=j_0}^{\infty}
\frac{(4\pi)^j }{j!}\int_{\R^2\backslash B_R(0)}({\bar{C}}_p|u|_{p^*})^{2j-p^*}|u|^{p^*}\mathrm{d}x
=\frac{1}{ {\bar{C}}_p^{p^*}|u|_{p^*}^{p^*}}\sum_{j=j_0}^{\infty}
\frac{(4\pi)^j }{j!}\int_{\R^2\backslash B_\varrho(0)}({\bar{C}}_p|u|_{p^*})^{2j}|u|^{p^*}\mathrm{d}x  \hfill\\
\ \ \ \  \leq   \frac{e^{4\pi({\bar{C}}_p|u|_{p^*})^2}}{ {\bar{C}}_p^{p^*}|u|_{p^*}^{p^*}}
    \int_{\R^2\backslash B_\varrho(0)} |u|^{p^*}\mathrm{d}x\leq \frac{e^{4\pi({\bar{C}}_p|u|_{p^*})^2}}{ {\bar{C}}_p^{p^*}}\leq
 \frac{e^{4\pi( \tilde{C}_p\|u\|)^2}}{ \bar{C}_p^{p^*}}  = \frac{e^{4\pi\tilde{C}_p^2 }}{ {\bar{C}}_p^{p^*}}.\hfill\\
\end{gathered}
\]
 From this inequality and \eqref{maintheorem1d}, we obtain that
\[
\mathbb{S}(\alpha)\leq C+{e^{4\pi\tilde{C}_p^2 }}{ {\bar{C}}_p^{-p^*}}<+\infty,~\forall \alpha\in(0,4\pi].
\]
The remaining parts are totally same as in \cite[Theorem 1.1]{Shen2}, so we omit them here.
\end{proof}

Now, we are able to verify that
 the variational functional $J_a:E\to\R$
defined by
$${J}_a (u)=\frac{1}{2}\int_{\R^2} [|\nabla u|^2+(A_1^2+A_2^2)u^2 ]\mathrm{d}x+\frac{1}p \int_{\R^2}a(x) |u|^p \mathrm{d}x
- \int_{\R^2} F (u)\mathrm{d}x,
$$
is well-defined and of class $\mathcal{C}^1(E,\R)$. Actually,
due to \eqref{definition} and $(f_1)$,
for all $\varepsilon>0$ and $\alpha>\alpha_0$, there is a constant $C_\varepsilon>0$ such that
\begin{equation}\label{growth1}
  |f(s)|\leq \varepsilon |s| +C_\varepsilon|s|^{q-1}\Phi_{\alpha,j_0}(s),~\forall s\in\R,
\end{equation}
where $\Phi_{\alpha,j_0}$ is defined by \eqref{Youngfunction} and $q>2$ can be arbitrarily chosen later. Using $(f_2)$, there holds
\begin{equation}\label{growth2}
  |F(s)|\leq \varepsilon |s|^{2}+C_\varepsilon|s|^{q }\Phi_{\alpha,j_0}(s),~\forall s\in\R.
\end{equation}
Moreover, without mentioned any longer,
let us exploit directly the following
inequality (see e.g. \cite[Lemma 2.1]{Yang0}):
\[
(\Phi_{\alpha,j_0}(s))^m\leq \Phi_{m\alpha,j_0}(s),~\forall s\in\R,~\alpha>0~\text{and}~m>1.
\]
With \eqref{growth1} and \eqref{growth2} in hands, exploiting Theorem \ref{maintheorem1},
 we could proceed as the calculations in \cite{Shen,Shen2}
to deduce that the variational functional $J_a$ associated with \eqref{mainequation1} is well-defined and belongs to $\mathcal{C}^1(E,\R)$
such that
\[
{J}_a' (u)[v]= \int_{\R^2} [ \nabla u\nabla v +(A_1^2+A_2^2+A_0)uv ]\mathrm{d}x+ \int_{\R^2}a(x) |u|^{p-2}uv \mathrm{d}x
- \int_{\R^2} f (u)v\mathrm{d}x,~\forall v\in E.
\]
In particular, it follows from \eqref{gauge0} that
\[
{J}_a' (u)[u]= \int_{\R^2} [ |\nabla u|^2 +3(A_1^2+A_2^2 )u^2 ]\mathrm{d}x+ \int_{\R^2}a(x) |u|^{p}\mathrm{d}x
- \int_{\R^2} f (u)u\mathrm{d}x.
\]
Hence, any (weak) solution of Eq. \eqref{mainequation1} corresponds to a critical point of $J_a$.
In order to search for the critical points of $J_a$, we introduce the following results.

\begin{lemma}\label{compact}
Let $1<p<2$ and $f$ satisfies \eqref{definition} and $(f_1)-(f_4)$. Suppose
there exists a sequence $\{u_n\}\subset E$ such that $u_n\rightharpoonup u$ in $E$
and $u_n\to u$ a.e. in $\R^2$.
If in addition, we assume that
\begin{equation}\label{compact1}
 \sup_{n\in \mathbb{N}}\int_{\R^2}f(u_n)u_n\mathrm{d}x\leq K_0
\end{equation}
for some $K_0\in(0,+\infty)$ independent of $n\in \mathbb{N}$,
then, going to a subsequence if necessary,
\begin{equation}\label{compact2}
 \lim_{n\to\infty}\int_{\Omega}F(u_n)\mathrm{d}x= \int_{\Omega}F(u)\mathrm{d}x~\text{\emph{for any compact set}}~\Omega\subset\R^2.
\end{equation}
Moreover, passing to a subsequence if necessary, there holds
\begin{equation}\label{compact3}
 \lim_{n\to\infty}\int_{\R^2}f(u_n)\psi\mathrm{d}x= \int_{\R^2}f(u)\psi\mathrm{d}x ~\text{\emph{for all}}~\psi\in C_0^\infty(\R^2).
\end{equation}
\end{lemma}

\begin{proof}
We can follow the essential ideas adopted in \cite[Lemma 2.1]{Figueiredo} and the details will be omitted.
\end{proof}

\section{The limit problem \eqref{mainequation2}}\label{Sec3}

The main objective of this section
is to investigate the existence of positive ground state solutions for the
CSS equation \eqref{mainequation2}
which acts as the ``limit problem" of Eq. \eqref{mainequation1}.

In order to solve Eq. \eqref{mainequation2}, we are going to look for the critical points of
its corresponding
variational functional
$J_\infty:E\to\R$ below
\begin{equation}\label{aafunctional}
{J}_\infty (u)=\frac{1}{2}\int_{\R^2} [|\nabla u|^2+(A_1^2+A_2^2)u^2 ]\mathrm{d}x+\frac{a_\infty}p \int_{\R^2} |u|^p \mathrm{d}x
- \int_{\R^2} F (u)\mathrm{d}x,~\forall u\in E.
\end{equation}
Arguing as before, it is simple to show that
$J_\infty$ is well-defined and it is of class
$\mathcal{C}^1(E,\R)$ satisfying
\[
{J}_\infty' (u)[v]= \int_{\R^2} [ \nabla u\nabla v +(A_1^2+A_2^2+A_0)uv ]\mathrm{d}x+\frac{a_\infty}p \int_{\R^2} |u|^{p-2}uv \mathrm{d}x
- \int_{\R^2} f (u)v\mathrm{d}x,~\forall v\in E.
\]
Moreover, we are derived from \eqref{gauge0} that
\[
{J}_\infty' (u)[u]= \int_{\R^2} [ |\nabla u|^2 +3(A_1^2+A_2^2 )u^2 ]\mathrm{d}x+ a_\infty\int_{\R^2}|u|^{p}\mathrm{d}x
- \int_{\R^2} f (u)u\mathrm{d}x.
\]

In what follows, we shall denote the Nehari manifold associated with $J_{\infty}$ by
$$
\mathcal{N}_{\infty}\triangleq\{u\in E\backslash\{0\}:J'_{\infty}(u)[u]=0\}
$$
and the corresponding ground state energy level on $\mathcal{N}_{\infty}$ is defined by
\begin{equation}\label{aa}
m_{\infty}\triangleq\min_{u\in \mathcal{N}_{\infty}}J_{\infty}(u).
\end{equation}

Our main result concerning the autonomous
CSS equation \eqref{mainequation2} is the following:

\begin{theorem} \label{nonautonomous1} Let $1<p<2$ and
suppose that $f$ satisfies \eqref{definition} and $(f_1)-(f_5)$, then Eq. \eqref{mainequation2} admits a
positive ground state solution $u_{\infty}\in E$ such that
$$
J_{\infty}(u_{\infty})=m_{\infty}=\inf_{u\in E\backslash\{0\}}
\max_{t\geq0}
J_{\infty}(tu ).
$$
\end{theorem}

The proof of the above theorem will be divided into several lemmas.
For simplicity, we shall always suppose that the nonlinearity $f$
satisfies \eqref{definition} and $(f_1)-(f_4)$ and do not mention them unless needed.

First of all, let us give some key observations on the shape
of the functional $J_\infty$.

\begin{lemma}\label{mountainpass1} Let $1<p<2$, then
there
 exists a constant $\zeta>0$ such that
\begin{equation}\label{mountainpass11}
m_\rho\triangleq \inf\big\{J_\infty(u):u\in E,\|u\| =\rho\big\}>0,~
\forall \rho\in(0,\zeta],
\end{equation}
and
\begin{equation}\label{mountainpass12}
n_\rho\triangleq \inf\big\{ J'_\infty(u)[u] :u\in E,
\|u\| =\rho\big\}>0,~
\forall \rho\in(0,\zeta].
\end{equation}
\end{lemma}

\begin{proof}
It follows from \eqref{growth2} that
\[
J_\infty(u)\geq \frac12\int_{\R^2} (|\nabla u|^2 +a_\infty|u|^p)\mathrm{d}x-\varepsilon\int_{\R^2}  |u|^2dx
 -C(\varepsilon,q,\alpha)\int_{\R^2}|u|^{q}\Phi_{\alpha,j_0}(u)dx.
\]
Using Lemma \ref{imbedding} and letting $\varepsilon>0$ be suitably small, with the help of \eqref{maintheorem1a},
 there exists a constant $\overline{\zeta}\in(0,1)$ such that
\[
J_\infty(u)\geq  \frac{1}{4} \|u\|^2
 -C\|u\|^q~\text{when}~\|u\|\leq \overline{\zeta}.
\]
In light of $q>2$, we can determine a constant $\zeta\in(0, \overline{\zeta})$ such that
 \eqref{mountainpass11} holds true. According to the definition of $J'_{\infty}$, it is easy to
reach \eqref{mountainpass12} as before. The proof is completed.
\end{proof}

\begin{lemma}\label{mountainpass2} Let $1<p<2$ and
suppose that $u\in E\backslash\{0\}$, then for all
 $t>0$, there holds
\[
J_\infty(tu)\to-\infty~ \text{as}~ t\to+\infty.
\]
In particular, the functional $J_\infty$ is not bounded from below.
\end{lemma}

\begin{proof}
For any fixed positive function $u\in E\backslash\{0\}$ and $t>1$, we have that
\[
\frac{J_\infty(tu)}{t^6}\leq\frac1p\int_{\R^2} [|\nabla u|^2+(A_1^2+A_2^2)u^2+a_\infty|u|^p ]\mathrm{d}x
- \frac{1}{ t^6}\int_{\R^2}F(t  u  ) dx.
\]
Due to $(f_4)$, one sees that $F(t)t^{-6}\to+\infty$ as $t\to+\infty$.
Thereby, using the Fatou's lemma,
we arrive at $J_\infty(tu)/t^6\to-\infty$ as $t\to+\infty$, and the claim follows.
\end{proof}

Relying on Lemmas \ref{mountainpass1} and \ref{mountainpass2},
we shall exploit the following critical point theorem
without the $(C)$ condition introduced in \cite{MW}
to find a $(C)$ sequence for $J_\infty$.

\begin{proposition}\label{mp}
Let $X$ be a Banach space and $\varphi\in \mathcal{C}^1(X,\R)$ Gateaux
differentiable for all $v\in X$, with G-derivative $\varphi^\prime(v)\in X^{-1}$ continuous from the norm
topology of $X$ to the weak $*$ topology of $X^{-1}$ and $\varphi(0) = 0$. Let $S$ be a closed subset of $X$ which
disconnects (archwise) $X$. Let $v_0 = 0$ and $v_1\in X$ be points belonging to distinct connected
components of $\bar{X}\backslash X$. Suppose that
\[
\inf_{S}\varphi\geq \varrho>0~\text{and}~\varphi(v_1)\leq0
\]
and let $\Gamma=\{\gamma\in C([0,1],X):\gamma(0)~\text{and}~\gamma(1)=v_1\}$. Then
$$c=\inf_{\gamma\in\Gamma}\max_{t\in[0,1]}\varphi(\gamma(t))\geq\varrho>0$$ and there is a $(C)_c$
sequence for $\varphi$.
\end{proposition}

Combining Lemmas \ref{mountainpass1} and \ref{mountainpass2} as well as Proposition
\ref{mp}, there is a sequence $\{u_n\}\subset E$ such that
\begin{equation}\label{sequence1}
  J_\infty(u_n)\to c_\infty~\text{and}~
(1+\|u_n\| )\|J_\infty^\prime(u_n)\|_{E^{-1}}\to0,
\end{equation}
 where
\begin{equation}\label{sequence2}
c_\infty\triangleq\inf_{\gamma\in\Gamma_{\infty}}\max_{t\in[0,1]}J_ \infty(\gamma(t))>0
\end{equation}
with $\Gamma_\infty=\{\gamma\in \mathcal{C}([0,1],E):\gamma(0)=0~\text{and}~J_\infty(\gamma(1))<0\}$.

\begin{lemma}\label{unique}
Let $1<p<2$,
then for every $u\in E\backslash\{0\}$, there exists a unique constant $t_u>0$
such that $J_\infty(t_uu)=\max\limits_{t\geq0}J_\infty(tu)$ and $t_uu\in \mathcal{N}_\infty$.
In particular, we could conclude that
$c_\infty=m_\infty=d_\infty$, where
$$
d_\infty\triangleq\inf_{u\in E\backslash\{0\}}
\max_{t\geq0}
J_\infty(tu ).
$$
\end{lemma}

\begin{proof}
For any $u\in E\backslash\{0\}$ and $t>0$, we define
$\xi(t)=J_\infty(tu)$ and so
\begin{align*}
  \xi'(t)=0 &\Longleftrightarrow
\int_{\R^2} [ t|\nabla u|^2 +3t^5(A_1^2+A_2^2 )u^2 ]\mathrm{d}x+ a_\infty t^{p-1}\int_{\R^2}  |u|^{p}\mathrm{d}x
- \int_{\R^2} f (tu)u\mathrm{d}x=0
 \\
    &   \Longleftrightarrow J'_\infty(tu)[tu]/t=0 \Longleftrightarrow J'_\infty(tu)[tu]=0 \Longleftrightarrow tu\in \mathcal{N}_\infty.
\end{align*}
Proceeding as in the proofs of Lemmas \ref{mountainpass1} and \ref{mountainpass2}, $\xi(t)$ possesses a critical point which corresponds to its
maximum, that is, there exists a constant $t_u> 0$ such that $\xi'(t_u)=0$. Next, we verify that $t_u$ is unique.
Arguing it indirectly, we would assume that there exist two constants $t_1, t_2 > 0$ with $t_1 \neq t_2$ such that $u_{t_i}\in \mathcal{N}_{\infty}$ for
$i \in\{1, 2\}$.
It follows from some elementary computations that
$$\begin{gathered}
{J}_\infty (t_1u)-{J}_\infty (t_2u)-\frac{t_1^6-t_2^6}{6t_1^6}{J}'_\infty (t_1u)[t_1u]\hfill\\
\ \ \ \  =\frac{t_1^2}{6}\left[2-3\left(\frac{t_2}{t_1}\right)^2+\left(\frac{t_2}{t_1}\right)^6\right]\int_{\R^2} |\nabla u|^2\mathrm{d}x\hfill\\
\ \ \ \ \ \   +\frac{t_1^p}{6p}\left[(6-p)-6\left(\frac{t_2}{t_1}\right)^p+p\left(\frac{t_2}{t_1}\right)^6\right]a_\infty\int_{\R^2} |u|^p\mathrm{d}x\hfill\\
\ \ \ \ \ \    +\int_{\R^2}\left[\frac{1-(t_1^{-1}t_2)^6}{6}f(t_1u)t_1u-F(t_1u)+F((t_1^{-1}t_2)t_1u)\right]\mathrm{d}x\hfill\\
\end{gathered}$$
and
$$\begin{gathered}
{J}_\infty (t_2u)-{J}_\infty (t_1u)-\frac{t_2^6-t_1^6}{6t_2^6}{J}'_\infty (t_2u)[t_2u]\hfill\\
\ \ \ \  =\frac{t_2^2}{6}\left[2-3\left(\frac{t_1}{t_2}\right)^2+\left(\frac{t_1}{t_2}\right)^6\right]\int_{\R^2} |\nabla u|^2\mathrm{d}x\hfill\\
\ \ \ \ \ \   +\frac{t_2^p}{6p}\left[(6-p)-6\left(\frac{t_1}{t_2}\right)^p+p\left(\frac{t_1}{t_2}\right)^6\right]a_\infty\int_{\R^2} |u|^p\mathrm{d}x\hfill\\
\ \ \ \ \ \    +\int_{\R^2}\left[\frac{1-(t_2^{-1}t_1)^6}{6}f(t_2u)t_2u-F(t_2u)+F((t_2^{-1}t_1)t_2u)\right]\mathrm{d}x.\hfill\\
\end{gathered}$$
Combining the above two formulas and ${J}'_\infty (t_iu)[t_iu]=0$ for $i\in\{1,2\}$, we arrive at
a contradiction if $t_1\neq t_2$. Next,
to coincide the three numbers with each other,
we shall firstly conclude that $c_{\infty}\leq d_{\infty}$, then $d_{\infty }\leq m_{\infty}$ and finally $m_{\infty}\leq c_{\infty}$
step by step.

\underline{\textbf{(1). $c_{\infty}\leq d_{\infty}$.}}
In view of Lemma \ref{mountainpass2}, there is a sufficiently large $t_0>0$ such that $J_{\infty}(t_0u)<0$
for every $u\in E\backslash\{0\}$. Define $\gamma_0(t)=tt_0u\in \Gamma_{\infty}$,
then $c_{\infty}\leq \max_{t\in[0,1]}J_{\infty}(tt_0u)\leq \max_{t\geq0}J_{\infty}(tu)$
which implies that $c_{\infty}\leq d_{\infty}$.

\underline{\textbf{(2). $d_{\infty}\leq m_{\infty}$.}}
We claim that, for all $u\in E$ and $t>0$, there holds
\begin{equation}\label{Nehari1}
\begin{gathered}
{J}_\infty ( u)-{J}_\infty (t u)-\frac{1-t ^6}{6 }{J}'_\infty ( u)[ u]
 =\frac{1}{6}\left(2-3 t^2+t^6\right)\int_{\R^2} |\nabla u|^2\mathrm{d}x\hfill\\
\ \ \ \ \ \   +\frac{a_\infty}{6p}\left[(6-p)-6t^p+pt^6\right]\int_{\R^2} |u|^p\mathrm{d}x
    +\int_{\R^2}\left[\frac{1-t^6}{6}f( u) u-F( u)+F( tu)\right]\mathrm{d}x.\hfill\\
\end{gathered}
\end{equation}
Due to $(f_2)$, it suffices to verify that
$$
\tau(s,t)\triangleq \frac{1-t^6}{6} f(s)s+ F(st)- F(s)\geq0 \quad
\text{for all}~s\geq0~\text{and}~t>0.
$$
Indeed, it is simple to calculate that
\begin{align*}
 \frac{\partial}{\partial t}\tau(s,t) &=f(st)s-t^{5}f(s)s
=t^{5}s^6\left[ \frac{f(st)}{(st)^5}-\frac{f(s)}{s^5} \right]  \\
   & \ \
\left\{
  \begin{array}{ll}
    \geq0, & \text{if}~t\in[1,+\infty), \\
    \leq0, & \text{if}~t\in(0,1],
  \end{array}
\right.
\end{align*}
where we have used $(f_2)$ in the last inequalities.
Therefore, we obtain that $t \mapsto\tau(s,t)$ is decreasing in
$(0, 1)$ and increasing in $(1, +\infty)$ for all $s\geq0$, respectively. It has that $\tau(s,t)\geq \min_{t\in(0,+\infty)} \tau(s,t)=\tau(s,1)=0$
  for every $s\geq0$ and the claim concludes. Owing to the claim, for all $u\in \mathcal{N}_{\infty}$, one sees that
 ${J}_\infty(u)\geq {J}_\infty(tu)$ for all $t>0$ yielding that $d_{\infty}\leq m_{\infty}$.

\underline{\textbf{(3). $m_{\infty}\leq c_{\infty}$.}}
We follow \cite[Theorem 4.2]{Willem} and present the details for the sake of completeness.
The manifold $\mathcal{N}_{\infty}$ clearly separates $E$
into two components, we say them by $\{J_{\infty}>0\}$ and $\{J_{\infty}<0\}$, respectively.
According to Lemma \ref{mountainpass1}, one can conclude that $\{J_{\infty}>0\}$ contains the origin and a small ball around the origin.
Moreover, adopting $(f_2)$, $J_{\infty}(u)\geq J_{\infty}(u)-\frac16J'_{\infty}(u)[u]\geq0$  for all $u\in \{J_{\infty}>0\}$,
so one must have that $\gamma(1)\in \{J_{\infty}<0\}$  for all $\gamma\in \Gamma_{\infty}$.
Because $\gamma\in \mathcal{C}^0$, there exists a $t_0\in(0,1)$ such that $\gamma(t_0)\in \mathcal{N}_{\infty}$
showing that $m_{\infty}\leq c_{\infty}$ by the arbitrariness of $\gamma$.
The proof is completed.
 \end{proof}

Since the nonlinearity $f$ fulfills the critical exponential growth at infinity which leads to the lack of compactness,
to recover it, we have to pull the mountain-pass level $c_\infty$ down below a critical value.
Have this aim in mind, inspired by \cite{AYA,AY,Cao,Figueiredo,DSD,Lu2012,Yang0},
we will consider the Moser
sequence functions defined by
\[
\bar{w}_n (x)\triangleq \frac{1}{ \sqrt{2\pi}}
\left\{
  \begin{array}{ll}
    \sqrt{\log n}, & \text{if}~0\leq |x|\leq \frac{1}{n}, \vspace{2mm}\\
    \frac{\log(\frac{1}{|x|})}{  \sqrt{\log n} }, & \text{if}~ \frac{1}{n}
    <|x|\leq 1, \vspace{2mm}\\
    0, & \text{if}~|x|>1,
  \end{array}
\right.
\]

\begin{lemma}\label{estimate}
Let $1<p<2$, then $0<c_\infty<\frac{2\pi}{\alpha_0}$.
\end{lemma}

\begin{proof}
We have $c_\infty\geq m_{\zeta}>0$ by Lemma \ref{mountainpass1}.
Taking advantage of Lemma \ref{mountainpass2}, it is not difficult to observe that
 $c_\infty=\inf_{\gamma\in \Gamma}\max_{t\in(0,1]}J_\infty(\gamma(t))
\leq \inf_{u\in E\backslash\{0\}}\max_{t>0}J_\infty(tu)$.
 As a consequence, it suffices to conclude that
there exists a function $w\in E\backslash\{0\}$ such that
$\max_{t>0}J_\infty(tw)< \frac{2\pi}{\alpha_0}$. It follows from some elementary computations that
$\overline{w}_n\in C_{0}^\infty(B_1(0))\subset E$ and it
satisfies
\[
\int_{\R^2}|\nabla\bar{w}_n|^2dx=\frac{1}{2\pi\log n}\int_{B_{1}(0)\backslash B_{\frac1n}(0)}
\frac{1}{|x|^2}dx=\frac{1}{ \log n}\int^{1}_{\frac1n} \frac{1}{\rho}d\rho=1,
\]
and
$$
\int_{\R^2}|\bar{w}_n|dx=\sqrt{\frac{\log n}{2\pi}}\frac{\pi}{n^2}
+\sqrt{\frac{2\pi}{\log n}}\left(\frac{1}{4}-\frac{\log n}{2n^2}-\frac{1}{4n^2}\right)=o_n(1).
$$
Thanks to the interpolation inequality, there holds $|\bar{w}_n|_{s}^{s}=o_n(1)$ for all $1\leq s<+\infty$.

Denoting $|\bar{w}_n|_{p}^{p}=\delta_n$ with $\delta_n\to0$ and we then define
$w_n=\bar{w}_n/(1+\sqrt[p]{\delta_n})$ for all $n\in \mathbb{N}$. Obviously, it has that $\|w_n\|\equiv1$ which together with $\delta_n\to0$ indicates that
\begin{equation}\label{estimate1a}
|\nabla w_n|_2^2\to1~\text{and}~|w_n|_p^p\to0.
\end{equation}
With the help of \eqref{gauge1} and \eqref{estimate1a}, we follow \cite[Lemma 3.10]{SSY} to show that
\begin{equation}\label{estimate1aa}
c(w_n)\triangleq \int_{\R^2}\left(A_1^2[w_n]+A_2^2[w_n]\right)w_n^2dx\to0.
\end{equation}
We now claim that there is a $n\in \mathbb{N}^+$ such that
\begin{equation}\label{estimate1}
\max_{t>0}J_\infty(tw_n)<\frac{2\pi}{\alpha_0}.
\end{equation}
Otherwise, for all $n\in \mathbb{N}^+$, there exists a $t_n>0$ corresponding to the maximum
point of
$\max_{t>0}J(tw_n)$
\begin{equation}\label{estimate2}
  J^\prime_\infty(t_nw_n)[t_nw_n]=0~\text{and}~
J_\infty(t_nw_n)= \max_{t>0}J_\infty(tw_n)\geq \frac{2\pi}{\alpha_0}.
\end{equation}
From $(f_3)-(f_4)$, for all $\epsilon\in(0, \beta_0)$, there exists a constant $R_\epsilon= R(\epsilon) > 0$ such
that
\[
f(s)s\geq M^{-1}_0(\beta_0-\epsilon)s^{\vartheta+1}e^{\alpha_0|s|^{2}},~\forall |s|\geq R_\epsilon.
\]
According to the second formula in \eqref{estimate2}, $\{t_n\}$ is bounded below by some positive constant.
For some sufficiently large $n\in \mathbb{N}$, one knows that $t_nw_n\geq R_\epsilon$ on $B_{1/n}(0)$.
Using \eqref{estimate2} again,
\[
\begin{gathered}
t_n^2|\nabla w_n|_2^2+a_\infty t_n^p | w_n|_p^p +3t_n^6c(w_n) =\int_{\R^2}f(t_nw_n)t_nw_ndx  \geq\int_{B_{1/n}(0)}f(t_nw_n)t_nw_ndx  \hfill\\
\ \ \ \   \geq \pi M^{-1}_0(\beta_0-\epsilon)(t_nw_n)^{\vartheta+1}e^{\alpha_0|t_nw_n|^{2}}
n^{-2}\hfill\\
\ \ \ \ =\pi M^{-1}_0(\beta_0-\epsilon)\bigg[ \frac{t_n\sqrt{\log n}}{(1+
\sqrt[p]{\delta_n})\sqrt{2\pi}}\bigg]^{\vartheta+1}\exp\bigg[ \alpha_0t_n^{2}\frac{\log n}{
2\pi(1+
\sqrt[p]{\delta_n})^{2}
}  -2\log n\bigg]
\hfill\\
\end{gathered}
\]
indicating that $\{t_n\}$ is uniformly bounded in $n\in \mathbb{N}$.
Up to a subsequence if necessary, there exists a constant $t_0\in (0,+\infty)$ such that
$t_n \to t_0$. Since $F(s)\geq0$ for all $s\in\R$, we invoke from \eqref{estimate1a} and
the second formula in
 \eqref{estimate2} that
\begin{equation}\label{estimate3}
 t_0^2\geq \frac{4\pi}{\alpha_0}.
\end{equation}
Choosing $\epsilon=\beta_0/2$, we apply \eqref{estimate1a}-\eqref{estimate1aa} and $t_n\to t_0$ to get
\[
(1-\vartheta)\log t_0\geq C\bigg[1+\frac{ \vartheta+1 }{2}\log (\log n)
+\big(\alpha_0t_0^{2}(2\pi)^{-1}-2\big)\log n\bigg]+o_n(1),
\]
where $C>0$ is independent of $n\in \mathbb{N}$. Recalling \eqref{estimate3},
we would arrive at a contradiction by tending $n\to\infty$
and so \eqref{estimate1} holds true. The proof is completed.
\end{proof}

\begin{lemma}\label{bounded}
Let $1<p<2$,
then each sequence $\{u_n\}\subset E$ satisfying \eqref{sequence1} is uniformly bounded in $E$.
Moreover, there is a constant $K_0>0$ independent of $n\in \mathbb{N}$
such that \eqref{compact1} holds true.
\end{lemma}

\begin{proof}
Given a sequence $\{u_n\}\subset E$ satisfying \eqref{sequence1},
it follows from $(f_2)$ that
\begin{align*}
c_\infty+o_n(1) & =J_\infty(u_n)-\frac16J'_\infty(u_n)[u_n] \\
    & \geq \frac13\int_{\R^2} |\nabla u_n|^2\mathrm{d}x+\left(\frac1p-\frac16\right)a_\infty\int_{\R^2} |u_n|^p\mathrm{d}x.
\end{align*}
Recalling Lemma \ref{estimate} and $1<p<2$, we see that $\{\|u_n\|\}$ is uniformly bounded in $n\in \mathbb{N}$. Then, we are derived from
$\|u_n\|  \|J_\infty^\prime(u_n)\|_{E^{-1}}\to0$ and \eqref{gauge1} that
\begin{align*}
\int_{\R^2} f (u_n)u_n\mathrm{d}x& = \int_{\R^2} [ |\nabla u_n|^2 +3(A_1^2+A_2^2 )u_n^2 ]\mathrm{d}x+ a_\infty\int_{\R^2}|u_n|^{p}\mathrm{d}x
 +o_n(1) \\
  & \leq\int_{\R^2} |\nabla u_n|^2 \mathrm{d}x+ a_\infty\int_{\R^2}|u_n|^{p}\mathrm{d}x+3 \bar{C}_r^2\|u_n\|^6
 +o_n(1)
\end{align*}
implying the desired result. The proof is completed.
\end{proof}

Before establishing the existence of ground state solutions for Eq. \eqref{mainequation2},
we shall introduce a new version type of Vanishing lemma with respect to our variational setting.

\begin{theorem}\label{Vanishing}
Let $1<p<2$ and $r>0$. If $\{u_n\}$ is bounded in $E$ and suppose that
\[
\limsup_{n\to\infty}\sup_{y\in\R^2}\int_{B_r(y)}|u_n|^{p} dx=0,
\]
 then $u_n\to 0$ in $L^s(\R^2)$ for all $p<s<+\infty$.
\end{theorem}
\begin{proof}
We follow the idea adopted in \cite[Lemma 1.21]{Willem}
to conclude the proof and exhibit it in detail for the convenience of the reader.
First of all, we recall the Gagliardo-Nirenberg inequality in \cite{Agueh} that
\begin{equation}\label{GN}
|u|_s^s\leq \mathcal{C}_s|\nabla u|_2^{s-p}|u|_p^p,~\forall u\in E~\text{and}~p<s<+\infty,
\end{equation}
where the constant $\mathcal{C}_s>0$ only depends on $s$.
So, we are derived form \eqref{GN} that
\begin{align*}
  \int_{B_r(y)}|u_n|^{s} dx & \leq \mathcal{C}_s\left(\int_{B_r(y)}|u_n|^{p} dx\right)
\left(\int_{B_r(y)}|\nabla u_n|^{2} dx\right)^{\frac{s-p}{2}}.
\end{align*}
Covering $\R^2$ by balls of radius $r$ in such a way that each point of $\R^2$ is contained in at most $3$ balls, we are able to see that
\[
\int_{\R^2}|u_n|^{s} dx\leq 3\mathcal{C}_s
\sup_{y\in\R^2}\left(\int_{B_r(y)}|u_n|^{p} dx\right)^{\frac{p(1-\varpi)}{q^*}}\|u_n\|^{s-p}.
\]
Under the assumption of this lemma, it holds that
$u_n\to0$ in $L^s(\R^2)$ for each $p<s<+\infty$.
The proof is completed.
\end{proof}

We are now in a position to show the proof of Theorem \ref{autonomous1}.

\begin{proof}[\emph{\textbf{Proof of Theorem \ref{nonautonomous1}}}]
Due to Lemmas \ref{mountainpass1}-\ref{mountainpass2} and Proposition \ref{mp},  there is a sequence $\{u_n\}\subset E$
satisfying
\eqref{sequence1}.
From Lemma \ref{bounded}, $\{\|u_n\|\}$ is uniformly bounded in $n\in \mathbb{N}$.
Passing to a subsequence if necessary, using Lemma \ref{imbedding},
 there exists a $u_\infty\in E$ such that $u_n\rightharpoonup u_\infty$ in $E$, $u_n\to u_\infty$ in $L^s_{\text{loc}}(\R^2)$
with $s\in(p,+\infty)$ and $u_n\to u_\infty$ a.e. in $\R^2$.
We claim that, there are $y\in\R^2$ and $r,\tau>0$ such that
\begin{equation}\label{claim}
 \int_{B_r(y)}|u_n|^p\mathrm{d}x\geq\tau.
\end{equation}
Otherwise, thanks to Theorem \ref{Vanishing},
we obtain that $u_n\to0$ in $L^s(\R^2)$ for every $s\in(p,+\infty)$.
According to Lemma \ref{bounded}, we now take a similar calculations in \eqref{compact2} to deduce that
\begin{equation}\label{Fn}
\lim_{n\to\infty}\int_{\R^2}F(u_n)dx=0.
\end{equation}
Our next goal is to show that
\begin{equation}\label{fn}
 \lim_{n\to\infty}\int_{\R^2}f(u_n)u_ndx=0.
\end{equation}
Indeed, on the one hand, taking \eqref{Fn},
$J_\infty(u_n)\to c_\infty$ and Lemma \ref{estimate} into account that
$\limsup\limits_{n\to\infty}|\nabla u_n|_2^2<\frac{4\pi}{\alpha_0}$.
Thereby, we shall choose $\alpha>\alpha_0$ sufficiently close to $\alpha_0$ and $\nu> 1$ sufficiently close to 1 in such
a way that $\frac1\nu + \frac1{\nu'}= 1$ with $\nu'>1$ and
$$
\alpha\nu|\nabla u_n|_2^2<4\pi(1-\epsilon)~\text{for some suitable}~\epsilon\in(0,1).
$$
We define
$$
\bar{u}_n=\sqrt{\frac{\alpha\nu}{4\pi(1-\epsilon)}} u_n,~\forall n\in \mathbb{N}.
$$
Obviously, one sees that $|\nabla \bar{u}_n|_2^2\leq 1$ for all sufficiently
$n\in \mathbb{N}$ and $|\bar{u}_n|_p^p$ is uniformly bounded in $n\in \mathbb{N}$.
On the other hand, we apply \eqref{maintheorem1a}
in \eqref{growth1} to get
\begin{align*}
\int_{\R^2}f(u_n)u_ndx & \leq \varepsilon \int_{\R^2}|u_n|^2dx+ C_\varepsilon\int_{\R^2}|u_n|^q\Phi_{\alpha,j_0}(u_n)dx \\
  & \leq  \varepsilon \int_{\R^2}|u_n|^2dx+ C_\varepsilon\left(\int_{\R^2}|u_n|^{q\nu'}dx\right)^{\frac1{\nu'}}
\left( \int_{\R^2}\Phi_{4\pi(1-\epsilon),j_0}(\bar{u}_n)dx\right)^{\frac1{\nu}}  \\
& \leq  \varepsilon \int_{\R^2}|u_n|^2dx+ C_\varepsilon \mathbb{S}(4\pi)\left(\int_{\R^2}|u_n|^{q\nu'}dx\right)^{\frac1{\nu'}}.
\end{align*}
Letting $n\to\infty$ and then tending $\varepsilon\to0^+$, we reach the desired result \eqref{fn}.
With \eqref{fn} in hands, as a direct consequence of $J_\infty'(u_n)[u_n]\to0$,
we derive that $|\nabla u_n|_2^2\to0$, $|u_n|_p^p\to0$ and $\int_{\R^2}(A_1^2+A_2^2)u_n^2dx\to0$.
Exploiting \eqref{Fn} and $J_\infty(u_n)\to c_\infty$
again, it immediately concludes that $c_\infty\equiv0$ which contradicts with $c_\infty>0$ in Lemma \ref{estimate}.
So, we see that \eqref{claim} must hold true.

According to \eqref{claim}, we define $v_n=u_n(\cdot+y_n)$ for every $n\in \mathbb{N}$.
Since both $J_\infty$ and $J'_\infty$ are translation invariant in $\R^2$,
one knows that $\{v_n\}$ is still a $(C)$ sequence of $J_\infty$ at the level $c_\infty$.
Arguing as before, passing to a subsequence if necessary, $v_n\rightharpoonup v$ in $E$, $v_n\to v$ in $L^s_{\text{loc}}(\R^2)$
with $s\in(p,+\infty)$ and $v_n\to v$ a.e. in $\R^2$. Moreover, we can see that $v\neq0$ by \eqref{claim}.
As a consequence, without loss of generality, we consider the sequence $\{u_n\}$ instead of $v_n$
to suppose that $u_\infty\neq0$. In view of \eqref{BLCS1} and \eqref{compact3}, there holds
$J'_\infty(u_\infty )=0$ and so $u_\infty\in \mathcal{N}_\infty$.
We are then derived from \eqref{sequence1} and the Fatou's lemma that
\begin{align}\label{equality}
\nonumber c_\infty & =\liminf_{n\to\infty} J_\infty(u_n)    =\liminf_{n\to\infty}\left\{  J_\infty(u_n)-\frac16J'_\infty(u_n)[u_n] \right\}\\
\nonumber &=\liminf_{n\to\infty}\left\{\frac13\int_{\R^2} |\nabla u_n|^2\mathrm{d}x+\left(\frac1p-\frac16\right)a_\infty\int_{\R^2} |u_n|^p\mathrm{d}x
+\frac16\int_{\R^2}[f(u_n)u_n-6F(u_n)]\mathrm{d}x \right\}\\
 \nonumber    & \geq \frac13\int_{\R^2} |\nabla u _\infty|^2\mathrm{d}x+\left(\frac1p-\frac16\right)a_\infty\int_{\R^2} |u _\infty|^p\mathrm{d}x
+\frac16\int_{\R^2}[f(u_\infty )u_\infty -6F(u_\infty )]\mathrm{d}x\\
&= J_\infty(u _\infty)-\frac16J'_\infty(u_\infty )[u _\infty]= J_\infty(u_\infty )\geq m_\infty.
\end{align}
From which, it follows from Lemma \ref{unique}
that $u_n\to u_\infty$ in $E$ along a subsequence.
In other words, we deduce that $u_\infty$
is a solution of Eq. \eqref{mainequation2} with $J_\infty(u_\infty)=m_\infty$.
The positivity of $u_\infty$ is trivial, and so we omit it here. The proof is completed.
\end{proof}

\begin{remark}
We invite the reader to see that Theorem \ref{maintheorem3}
is a direct corollary of Theorem \ref{nonautonomous1}.
\end{remark}

\section{Proof od Theorem \ref{maintheorem2}}\label{Sec4}

In this section, we are going to investigate the existence of positive ground state solutions for Eq. \eqref{mainequation1}.
From the view point of variational method, we search for critical points of
$J_a$ defined in \eqref{aafunctional}.
Recalling the discussions in Section \ref{Sec2},
$J_a$ is well-defined and of class of $C^1(E,\R)$.

We shall prove the following result.

\begin{theorem}\label{autonomous1}
 Let $1<p<2$ and suppose $(A_1)-(A_2)$. If $f$ satisfies \eqref{definition} and $(f_1)-(f_4)$, then
Eq. \eqref{mainequation1} admits at least a positive ground state solution $u_a\in E$ such that
$$
J_a(u_a )=m_{a}=\inf_{u\in E\backslash\{0\}}
\max_{t\geq0}
J_a(tu ).
$$
\end{theorem}

Associated with $J_a$, we have the Nehari manifold given by
$$
\mathcal{N}_a=\{u \in E\setminus \{0\}\,:\, J'_a(u)[u]=0\},
$$
and define the minimization problem
\begin{equation} \label{aaa}
	m_a=\min_{u \in \mathcal{N}_a}J_a(u).
\end{equation}
By definitions of $m_\infty$ and $m_a$ in \eqref{aa} and \eqref{aaa}, adopting $({A}_2)$, it is easy to check that
\begin{equation} \label{compare}
m_a<m_\infty.
\end{equation}

Moreover, because $a$ is a positive and bounded function, it permits us to repeat
 the arguments in Section \ref{Sec3}
to find a $(C)$ sequence $\{u_n\}\subset E$ of $J_a$ at the level $c_a$, where
\begin{equation}\label{mplevel}
 c_a\triangleq\inf_{\gamma\in\Gamma_{a}}\max_{t\in[0,1]}J_ a(\gamma(t))>0
\end{equation}
with $\Gamma_a=\{\gamma\in \mathcal{C}([0,1],E):\gamma(0)=0~\text{and}~J_a(\gamma(1))<0\}$.
We can also deduce that
\begin{equation} \label{aaaaaa}
	m_a=c_a=d_a\triangleq \inf_{u\in E\backslash\{0\}}
\max_{t\geq0}
J_a(tu )
\end{equation}
and
\begin{equation} \label{aaaaaaaa}
	0<c_a <\frac{2\pi}{\alpha_0}.
\end{equation}

Actually, the essential, or unique, difference between the proof of Theorem \ref{nonautonomous1}
and that of Theorem \ref{autonomous1} is that whether the variational functional is translation invariant
in $\R^2$. Clearly, we realize that $J_a$ does not have such a good property because of
the appearance of the nonconstant potential $a$. Thereby,
to reach the proof, it is enough to verify that the weak limit of $(C)_{c_a}$ sequence is nontrivial.

\begin{lemma}\label{nontrivial}
Under the assumptions of Theorem \ref{autonomous1}, if $\{u_n\}\subset \mathcal{N}_a$ denotes a $(C)_{c_a}$ sequence of $J_a$
and $u_n\rightharpoonup  u_a$ in $E$ along a subsequence, then
$u_a\neq0$.
\end{lemma}

\begin{proof} Assume by contradiction that $u_a=0$ and let $t_n>0$ such that $t_nu_n \in \mathcal{N}_\infty$
by Lemma \ref{unique}. The standard calculations show that
$\{t_n\}$ is bounded, and so,
$$
c_a +o_n(1)=J_a (u_n)= \max_{t \geq 0}J_a (tu_n) \geq J_a (t_nu_n)=J_\infty(t_nu_n)
+\frac{t_n^p}{p}\int_{\R^2}(a(x)-a_\infty)|u_n|^p dx,
$$
where we have used  Lemma \ref{unique} again in the second equality.
From this,
$$
c_a  +o_n(1) \geq m_\infty + \frac{t_n^p}{p}\int_{\R^2}(a(x)-a_\infty)|u_n|^pdx.
$$
In light of the fact that $\{t_n\}$ is bounded, we can make use of $(A_2)$ to have that
$$
\int_{\R^2}(a(x)-a_\infty)|u_n|^pdx\to 0
$$
loading to
$$
c_a \geq m_\infty,
$$
which contradicts with \eqref{compare} and \eqref{aaaaaa}. The proof is completed.
\end{proof}

\begin{proof}[\emph{\textbf{Proof of Theorem \ref{autonomous1}}}]
Owing to \eqref{aaaaaa}, we just need to find a sequence $\{u_n\}\subset \mathcal{N}_a$
and it is a $(C)_{m_a}$ sequence of $J_a$. It is standard, we refer the reader to \cite[Theorem 1.1]{PSZZ}
and so the proof is done.
\end{proof}

\section{Proof od Theorem \ref{maintheorem4}}\label{Sec5}

In this section, we aim to derive that Eq. \eqref{mainequation2}
admits a mountain-pass type solution whose energy is equal to the mountain-pass level.

The main result in this direction can be stated as follows.

\begin{theorem}\label{mpsolution}
Let $1<p<2$ and suppose $(A_1)-(A_2)$. If $f$ satisfies \eqref{definition} and $(f_1)-(f'_2)$ as well as $(f_3)-(f_4)$, then
Eq. \eqref{mainequation2} has a positive mountain-pass type solution in $u\in E$ with $J_\infty(u)=c_ \infty$,
where $J_\infty$ and $c_\infty$ are defined by \eqref{aafunctional} and \eqref{sequence2}, respectively.
\end{theorem}

As we have pointed out in the Introduction,
when $(f_2)$ is absence, we cannot restore the compactness as what we have done in the Sections \ref{Sec3}
and \ref{Sec4}. Speaking it clearly, let $\{u_n\}\subset E$
be a $(C)_{c_\infty}$ sequence of $J_\infty$, it is impossible to conclude that $\{u_n\}$
admits a strongly convergent subsequence in $E$ by \eqref{equality}.
The existence of such a sequence is guaranteed by adopting some very similar calculations in Section \ref{Sec3}.

Whereas, since the whole space $\R^2$ itself also results in the lack of compactness,
  we shall always restrict ourselves in the radially
symmetric subspace of $E$. In other words,
in this section, we prefer to take advantage of $E_r$ to be the work space, instead of $E$.

Now, we are able to verify that the variational functional $J_\infty$ satisfies the $(C)$ condition at the level $c_\infty$.

\begin{lemma}\label{Cccondition}
Under the assumptions of Theorem \ref{mpsolution},
if $\{u_n\}\subset E_r$ is a $(C)_{c_\infty}$ sequence of $J_\infty$, then
there is a $u_0\in E$ such that $u_n\to u_0$ in $E_r$ along a subsequence.
\end{lemma}

\begin{proof}
Proceeding as the proof of  Lemma \ref{bounded}, $\{\|u_n\|\}$ is uniformly bounded in $n\in \mathbb{N}$.
Passing to a subsequence if necessary,
there is a
  $u_0\in E_r$ such that $u_n\rightharpoonup u_0$ in $E_r$,
$u_n\to u_0$ in $L^s(\R^N)$ with $s>p$ and $u_n\to u_0$ a.e. in $\R^N$.
Combining \eqref{BLCS1} and \eqref{compact3}, one has $J^\prime_\infty(u_0)=0$  which implies that
\begin{equation}\label{Cc1}
J_\infty(u_0)=J_\infty(u_0)-\frac{1}{6}J'_\infty(u_0)[u_0]\geq0.
\end{equation}
Moreover, it follows from the Br\'{e}zis-Lieb lemma, \eqref{BLCS2}, \eqref{compact2} and \eqref{Cc1} that
\begin{align*}
  c_\infty & =\frac12|\nabla u_n|_2^2+\frac{a_\infty}p| u_n|_p^p +\frac12\int_{\R^2}(A_1^2[u_n]+A_2^2[u_n])u_n^2dx
-\int_{\R^2}F(u_n)dx+o_n(1)\\
    & =\frac12|\nabla u_n-\nabla u_0|_2^2+\frac{a_\infty}p| u_n-u_0|_p^p +\frac12\int_{\R^2}(A_1^2[u_n-u_0]+A_2^2[u_n-u_0])(u_n-u_0)^2dx\\
&\ \ \ \  +\frac12|\nabla u_0|_2^2+\frac{a_\infty}p| u_0|_p^p +\frac12\int_{\R^2}(A_1^2[u_0]+A_2^2[u_0])u_0^2dx
-\int_{\R^2}F(u_0)dx+o_n(1)\\
&\geq \frac12|\nabla u_n-\nabla u_0|_2^2+J_\infty(u_0)+o_n(1)\geq \frac12|\nabla u_n-\nabla u_0|_2^2 +o_n(1).
\end{align*}
From which and Lemma \ref{estimate}, then $\limsup\limits_{m\to\infty}|\nabla u_n-\nabla u_0|_2^2<\frac{4\pi}{\alpha_0}$.
Consequently, we shall choose $\alpha>\alpha_0$ sufficiently close to $\alpha_0$ and $\nu> 1$ sufficiently close to 1 in such
a way that $\frac1\nu + \frac1{\nu'}= 1$ with $\nu'>1$ and
$$
\alpha\nu|\nabla u_n-\nabla u_0|_2^2<4\pi(1-\hat{\epsilon})~\text{for some suitable}~\epsilon\in(0,1).
$$
We define
$$
\hat{u}_n=\sqrt{\frac{\alpha\nu}{4\pi(1-\hat{\epsilon})}}( u_n-u_0),~\forall n\in \mathbb{N}.
$$
Obviously, one sees that $|\nabla \hat{u}_n|_2^2\leq 1$ for all sufficiently
$n\in \mathbb{N}$ and $|\hat{u}_n|_p^p$ is uniformly bounded in $n\in \mathbb{N}$.
Besides, for the above fixed $\hat{\epsilon}\in(0,1)$, we need the following two types of Young's inequality
\[
 |a+b|^2\leq(1+\hat{\epsilon})|a|^2+(1+\hat{\epsilon}^{-1})|b|^2,~\forall a,b\in\R
\]
and
\[
e^{a+b}-d\leq \frac{1}{1+\hat{\epsilon}}\big[e^{(1+\hat{\epsilon})a}-d\big]
+\frac{\hat{\epsilon}}{1+\hat{\epsilon}}\big[e^{(1+\hat{\epsilon}^{-1})b}-d\big],~\forall a,b,d\in\R.
\]
By means of the above facts together with \eqref{maintheorem1a}, we derive
 \begin{align*}
\int_{\R^2}\Phi_{\alpha\nu,j_0}(u_n)dx &\leq\frac{1}{1+\bar{\epsilon}}
\int_{\R^2} \Phi_{4\pi(1+\bar{\epsilon})^{-2},j_0} (\hat{u}_n)  dx
+\frac{\bar{\epsilon}}{1+\bar{\epsilon}}\int_{\R^2} \Phi_{\nu\alpha(1+\bar{\epsilon}^{-1})^2,j_0} (u_0) dx \\
  & \leq \frac{\mathbb{S}(4\pi)}{1+\bar{\epsilon}}+\frac{C_2\bar{\epsilon}}{1+\bar{\epsilon}}\leq C_3 <+\infty,~\forall n\in \mathbb{N}^+.
\end{align*}
As a consequence, by \eqref{growth1}, we obtain
\begin{align}\label{Cc2}
 \nonumber \bigg|\int_{\R^2}f(u_n)(u_n-u_0)dx\bigg| & \leq \bigg(\int_{\R^2}|u_n|^2dx\bigg)^{\frac12}
\bigg(\int_{\R^2}|u_n-u_0|^2dx\bigg)^{\frac{1}{2}}\\
\nonumber   & \ \ \ \   + C  |u_n|_{2(q-1)\nu^\prime}  ^{(q-1)}
 |u_n-u_0|_{2\nu^\prime}
\bigg(\int_{\R^2}\Phi_{\alpha\nu,j_0}(u_n)dx\bigg)^{\frac{1}{\nu}}\\
&=o_n(1).
\end{align}
It is simple to see that
\begin{equation}\label{Cc3}
\int_{\R^2}f(u_0)(u_n-u_0)dx=o_n(1).
\end{equation}
For all $1<r<2$, thanks to the significant inequality \cite[(2.2)]{Simon} which can be stated as follows
\[
(|y_2|^{r-2}y_2-|y_1|^{r-2}y_1)\cdot(y_2-y_1)\geq \hat{C}_r\cdot
 \frac{|y_2-y_1|^2}{(|y_2|+|y_1|)^{2-r}}.
\]
From which, using $J_\infty^\prime(u_n)=o_n(1)$ and $J _\infty^\prime(u_0)=0$ as well as \eqref{BLCS1} and \eqref{Cc2}-\eqref{Cc3}, it holds that
\begin{align*}
  o_n(1) &= J_\infty ^\prime(u_n)[u_n-u_0]-  J _\infty^\prime(u_0)[u_n-u_0] \\
   & = \int_{\R^2}\big[ |\nabla u_n-\nabla u_0 |^{2}+
(| u_n|^{p-2} u_n-|u _0 |^{p-2} u_0)( u_n- u_0)\big]dx\\
&\ \ \ \  +\int_{\R^2}(A_1^2[u_n]u_n+A_2^2[u_n]u_n)(u_n-u)dx  +\int_{\R^2}(A_1^2[u_0]u_0+A_2^2[u_0]u_0)(u_n-u_0)dx \\
&\ \ \ \  +\int_{\R^2}f(u_n)(u_n-u_0)dx
+\int_{\R^2}f(u_0)(u_n-u_0)dx\\
&=\int_{\R^2}\big[ |\nabla u_n-\nabla u_0 |^{2}+
(| u_n|^{p-2} u_n-|u _0 |^{p-2} u_0)( u_n- u_0)\big]dx
+o_n(1)\\
&\geq o_n(1)
\end{align*}
yielding that
\[
 |\nabla u_n-\nabla u_0 |^{2}_2=o_n(1)~\text{and}~
\int_{\R^2}
(| u_n|^{p-2} u_n-|u _0 |^{p-2} u_0)( u_n- u_0) dx.
\]
At this stage, we apply the H\"{o}lder's inequality to get
\begin{align*}
 \int_{\R^2}|u_n-u_0|^pdx &\leq \hat{C}_p^{-\frac{p}{2}}\int_{\R^2}|(|u_n|^{p-2}u_n-|u _0|^{p-2}u _0)(u_n-u_0)|^{\frac{p}{2}}(|u_n|+|u_0|)^{\frac{p(2-p)}{2}}dx \\
  & \leq \hat{C}_p^{-\frac{p}{2}}\bigg(\int_{\R^2}|(|u_n|^{p-2}u_n-|u _0|^{p-2}u_0 )(u_n-u_0)|dx\bigg)^{{\frac{p}{2}}}
\bigg( \int_{\R^2}(|u_n|+|u_0|)^{p}dx
\bigg)^{{\frac{2-p}{2}}}\\
&\leq C\bigg(\int_{\R^2}|(|u_n|^{p-2}u_n-|u _0|^{p-2}u_0 )(u_n-u_0)|dx\bigg)^{{\frac{p}{2}}}=o_n(1).
\end{align*}
Thus, we can derive that $u_n\to u_0$ in $E_r$ as $n\to\infty$.
The proof is completed.
\end{proof}

\begin{proof}[\emph{\textbf{Proof of Theorem \ref{mpsolution}}}]
In view of Section \ref{Sec3}, there is a $(C)_{c_\infty}$ sequence $\{u_n\}\subset E_r$
 of $J_\infty$. So, we can finish the proof by Lemma \ref{Cccondition}.
This proof also concludes Theorem \ref{maintheorem4}.
\end{proof}

\bigskip

\end{document}